\documentclass[12pt,a4paper,oneside,reqno]{article}

\usepackage[utf8]{inputenc}
\usepackage[english]{babel}
\usepackage[left=1in,right=1in,top=1in,bottom=1in]{geometry}
\usepackage{amsfonts,amssymb,amsmath,amsthm}
\usepackage[symbol]{footmisc}
\usepackage{longtable}
\usepackage{enumerate}
\usepackage{cite}
\usepackage{hyperref}
\usepackage{fancyhdr}
\theoremstyle{plain}
\newtheorem{theorem}{Theorem}
\newtheorem{lemma}[theorem]{Lemma}
\newtheorem{proposition}[theorem]{Proposition}
\newtheorem{corollary}[theorem]{Corollary}
\theoremstyle{definition}
\newtheorem{definition}[theorem]{Definition}
\newtheorem{remark}[theorem]{Remark}

\DeclareMathOperator{\Id}{\rm Id}
\DeclareMathOperator{\Aut}{\rm Aut}
\DeclareMathOperator{\tr}{\rm tr}

\usepackage{fouriernc}

\begin{document}

\bigskip

\noindent{\Large
Classification of
$3$-dimensional unital commutative algebras}\footnote{
The work is supported by
FCT 2023.08031.CEECIND and UID/00212/2025. The original concept, theoretical framework, and primary results were developed by M.E. and I.K.; 
while Aristotle (https://aristotle.harmonic.fun/)   acted as a non-human research assistant, assisting in error verification and proof simplification.}

\bigskip

\begin{center}
{\bf
Maqpal Eraliyeva\footnote{Romanovsky Institute of Mathematics,   Uzbekistan; \ eraliyevamaqpal@gmail.com}
\&
Ivan Kaygorodov\footnote{CMA-UBI, University of Beira Interior,   Portugal; \ kaygorodov.ivan@gmail.com}}
\end{center}

\noindent {\bf Abstract:}
{\it We present the complete classification of non-isomorphic complex $3$-dimensional unital commutative algebras.
Every such algebra decomposes canonically as $\mathbb Ce_1\oplus V$, where $e_1$ is the unit and $V$ is the kernel of the trace form; the induced traceless $2$-dimensional commutative algebra on $V$ is one of four algebras, and the classification reduces to the classification of symmetric bilinear forms on $V$ up to the automorphisms of this kernel.
The answer consists of six algebras, three one-parameter families, one two-parameter family and one three-parameter family.}

\medskip

\noindent {\bf Keywords}:
{\it unital algebra, commutative algebra, algebraic classification.}

\medskip

\noindent {\bf MSC2020}:
17A30 (primary);
17A99 (secondary).

\medskip

\tableofcontents

\section*{Introduction}
The algebraic and geometric classifications of $n$-dimensional algebras from a certain variety defined by a certain family of polynomial identities are classic problems in the theory of non-associative algebras \cite{akt,   FKS25, FKS25-2, k23, MS}. The complete classification of $n$-dimensional algebras is given only for $n=2$ (see \cite{Volkov2019, P00})\footnote{The list of \cite[Theorem~3]{Volkov2019} is not complete: the two-parameter family ${\bf D}_4(\alpha,\beta)$: $e_1e_1=e_1$, $e_1e_2=(1-\alpha)e_1+\beta e_2$, $e_2e_1=\alpha e_1+(1-\beta)e_2$ of pairwise non-isomorphic algebras of class ${\bf D}$ is missing (it is lost in the case $c_{12}^1+c_{21}^1\neq0$, $c_{12}^2+c_{21}^2=1$ of the proof of \cite[Lemma~7]{Volkov2019}); in particular, the commutative algebra ${\bf D}_4(\frac12,\frac12)$ is missing in the list of commutative algebras. Moreover, the algebras ${\bf E}_3(\alpha,\beta,-1)$ and ${\bf E}_3(\beta,\alpha,-1)$ from that list are isomorphic. The results of the present paper do not depend on the list of \cite{Volkov2019}, see Lemma~\ref{lem:traceless} and Remark~\ref{rem:oldkernels}.}. Namely, over the complex field, all $2$-dimensional algebras are described as the union of a family of algebras with $4$ infinite parameters and some families with fewer parameters.
It is easy to understand that in the case of $3$-dimensional algebras, the complete classification of complex algebras can be described as the union of a family of algebras with $18$ infinite parameters and some families with fewer parameters. This is one of the reasons why we still do not have a complete classification of complex $3$-dimensional algebras. Nevertheless, some particular cases of $3$-dimensional algebras have been classified.
Classifications are known for $3$-dimensional:
anticommutative \cite{IKV19},
$\delta$-Novikov \cite{AKL},
semi-alternative  \cite{AANS},
transposed Poisson \cite{BFK23},
commutative post-Lie \cite{AAK},
and evolution \cite{ccsmv} algebras;
as well as $3$-dimensional nil-algebras \cite{KS25},
$3$-dimensional right-alternative and noncommutative Jordan superalgebras \cite{AKK,B09},
$3$-dimensional ${\rm Hom}$-Lie algebras \cite{FR23},
$3$-dimensional simple real Bol algebras \cite{FFKMN},
unital $3$-dimensional structurable algebras \cite{AEK},
and so on.
Our work was inspired by a recent study of $3$-dimensional simple unital commutative algebras given by Glizburg, Onerbek, and Pchelintsev \cite{GOP25}.

\medskip 

A $3$-dimensional commutative algebra with unit $e_1$ is determined by the products of the elements of a complement $W$ of $\mathbb Ce_1$; these products define a symmetric bilinear form on $W$ and a $2$-dimensional commutative algebra structure on $W$ (the \emph{non-unital kernel}).
The isomorphism class of the non-unital kernel depends on the choice of $W$; therefore we fix $W$ canonically, as the kernel of the linear form $x\mapsto\tr(L_x)$.
With respect to this complement the non-unital kernel is traceless, and traceless $2$-dimensional commutative algebras correspond to binary cubic forms; hence there are exactly four of them (Lemma~\ref{lem:traceless}).
The main result of the paper (Theorem~\ref{thm:main}) gives a description of complex $3$-dimensional unital commutative algebras 

\section{Preliminaries}\label{sec:prel}

Throughout the paper all algebras and vector spaces are over the complex field $\mathbb C$.
Let $\mathcal A$ be a $3$-dimensional unital commutative algebra.
The unit of $\mathcal A$ will always be denoted by $e_1$.
In multiplication tables we write only the products $e_ie_j$ with $2\le i\le j\le 3$;
the products $e_1e_i=e_ie_1=e_i$, the products obtained by commutativity, and all zero products are omitted.
For $x\in\mathcal A$ let $L_x\colon y\mapsto xy$ be the operator of multiplication by $x$ and put
\[
\tau(x):=\tr(L_x).
\]
Then $\tau\colon \mathcal A\to\mathbb C$ is a linear form and $\tau(e_1)=\tr(\Id)=3$.

\subsection{The canonical non-unital kernel}

Let $W$ be a complement of $\mathbb Ce_1$ in $\mathcal A$, i.e.\ $\mathcal A=\mathbb Ce_1\oplus W$.
For $x,y\in W$ we can write uniquely
\begin{equation}\label{eq:decomp}
xy=\varphi_W(x,y)\,e_1+\mu_W(x,y),\qquad \varphi_W(x,y)\in\mathbb C,\quad \mu_W(x,y)\in W .
\end{equation}
Clearly, $\varphi_W$ is a symmetric bilinear form on $W$ and $\mu_W$ is a commutative multiplication on $W$.
The $2$-dimensional commutative algebra $(W,\mu_W)$ is called the \emph{non-unital kernel} of $\mathcal A$ with respect to $W$.
Choosing a basis $e_2,e_3$ of $W$ in which $\mu_W$ is written in a normal form of the classification of $2$-dimensional commutative algebras \cite{Volkov2019}, one obtains a basis of $\mathcal A$ in which the structure constants $\alpha_{ij}^k$ with $2\le i,j,k\le 3$ are the structure constants of that normal form. 
Note that (see Remark~\ref{rem:oldkernels})
 the list of commutative algebras in \cite{Volkov2019} misses the algebra 
\begin{center}
${\bf D}_4(\frac12,\frac12)$:\quad  
$e_2e_2=e_2$, \quad 
$e_2e_3=\frac12e_2+\frac12e_3$.
\end{center}
However, \emph{the isomorphism class of $(W,\mu_W)$ depends on the choice of $W$.}
Indeed, let $a\colon W\to\mathbb C$ be a linear form and $W_a:=\{x+a(x)e_1 : x\in W\}$.
Identifying $W_a$ with $W$ by $x\mapsto x+a(x)e_1$, a direct computation gives
\begin{equation}\label{eq:shift}
\mu_{W_a}(x,y)=\mu_W(x,y)+a(x)y+a(y)x,\qquad
\varphi_{W_a}(x,y)=\varphi_W(x,y)-a\bigl(\mu_W(x,y)\bigr)-a(x)a(y).
\end{equation}
For example, the algebra with the table $e_2e_2=e_2+e_3,\ e_2e_3=\frac12 e_3$ has the non-unital kernel ${\bf A}_1(\frac12)$ with respect to $W=\langle e_2,e_3\rangle$, while in the basis $e_1,\ e_2-\frac12e_1,\ e_3$ its table is $e_2e_2=\frac14e_1+e_3$, i.e.\ its non-unital kernel with respect to $W_a$, $a=-\frac12 e_2^*$, is ${\bf A}_3$.
Consequently, a case-by-case study of the possible non-unital kernels with respect to an \emph{arbitrary} complement produces many repetitions (see Remark~\ref{rem:oldkernels} below).
To avoid this, we fix the complement canonically.

\begin{lemma}\label{lem:canonical}
Let $\mathcal A$ be a $3$-dimensional unital commutative algebra and let $W$ be a complement of $\mathbb Ce_1$.
\begin{enumerate}
\item[$({\rm A})$] For every $x\in W$ we have $\tau(x)=\tr\bigl(\mu_W(x,\cdot)\bigr)$.
\item[$({\rm B})$] $V(\mathcal A):=\ker\tau$ is a complement of $\mathbb Ce_1$, and it is the unique complement $W$ such that the algebra $(W,\mu_W)$ is \emph{traceless}, i.e.\ $\tr\bigl(\mu_W(x,\cdot)\bigr)=0$ for all $x\in W$.
\item[$({\rm C})$] Conversely, let $V$ be a $2$-dimensional space endowed with a traceless commutative multiplication $\mu$ and a symmetric bilinear form $\varphi$. Then the space $\mathcal A(V,\mu,\varphi):=\mathbb Ce_1\oplus V$, with the multiplication for which $e_1$ is the unit and $xy=\varphi(x,y)e_1+\mu(x,y)$ for $x,y\in V$, is a $3$-dimensional unital commutative algebra with $V(\mathcal A(V,\mu,\varphi))=V$.
\item[$({\rm D})$] $\mathcal A(V,\mu,\varphi)\cong\mathcal A(V',\mu',\varphi')$ if and only if there exists a linear bijection $g\colon V\to V'$ such that
\[
g\bigl(\mu(x,y)\bigr)=\mu'\bigl(g(x),g(y)\bigr)\quad\text{and}\quad \varphi(x,y)=\varphi'\bigl(g(x),g(y)\bigr)\qquad\text{for all } x,y\in V .
\]
\end{enumerate}
\end{lemma}

\begin{proof} We consider the following four statements below.
\begin{enumerate}
    \item[$({\rm A})$] Let $w_2,w_3$ be a basis of $W$ and $x\in W$. Since $L_x e_1=x\in W$, the $e_1$-coordinate of $L_xe_1$ is zero, and by \eqref{eq:decomp} the $W$-component of $L_xw_i$ is $\mu_W(x,w_i)$. Hence \begin{center}
        $\tau(x)=\tr(L_x)=\tr(\mu_W(x,\cdot))$.
    \end{center}

\item[$({\rm B})$] Since $\tau(e_1)=3\neq0$, $\ker\tau$ is a $2$-dimensional subspace not containing $e_1$, i.e.\ a complement of $\mathbb Ce_1$. By ({\rm A}), $(W,\mu_W)$ is traceless if and only if $W\subseteq\ker \tau$, i.e.\ if and only if $W=\ker\tau$.

\item[$({\rm C})$] The multiplication is obviously commutative and $e_1$ is its unit; by ({\rm A}) applied to $W=V$ we get $\tau|_V=0$, i.e.\ $V=\ker\tau$.

\item[$({\rm D})$] Let $\psi\colon\mathcal A(V,\mu,\varphi)\to\mathcal A(V',\mu',\varphi')$ be an isomorphism. Then $\psi(e_1)=e_1$ because the unit is unique, and $L_{\psi(x)}=\psi L_x\psi^{-1}$ gives $\tau'(\psi(x))=\tau(x)$; hence $\psi(V)=V'$ by ({\rm B}) and ({\rm C}). Put $g=\psi|_V$. For $x,y\in V$ we have
\begin{center}
$\psi(xy)=\varphi(x,y)e_1+g(\mu(x,y))$ and $\psi(x)\psi(y)=\varphi'(g(x),g(y))e_1+\mu'(g(x),g(y))$;
\end{center} comparing the components we obtain the required identities.
Conversely, such a map $g$ extended by $\psi(e_1)=e_1$ is an isomorphism.
\end{enumerate}\end{proof}

\begin{definition}
The $2$-dimensional commutative algebra $(V(\mathcal A),\mu_{V(\mathcal A)})$ is called the \emph{canonical non-unital kernel} of $\mathcal A$.
We say that $\mathcal A$ has the canonical non-unital kernel of type $\mathcal G$ if $(V(\mathcal A),\mu_{V(\mathcal A)})\cong\mathcal G$.
\end{definition}

By Lemma~\ref{lem:canonical}, the type of the canonical non-unital kernel is an isomorphism invariant of $\mathcal A$, and the classification of $3$-dimensional unital commutative algebras reduces to the following two steps:
\begin{enumerate}
\item[(i)] classify the traceless $2$-dimensional commutative algebras $(V,\mu)$ up to isomorphism and compute their automorphism groups;
\item[(ii)] for every such $(V,\mu)$, classify the symmetric bilinear forms $\varphi$ on $V$ up to the action $\varphi\mapsto\varphi(g\,\cdot,g\,\cdot)$ of the group $\Aut(V,\mu)$.
\end{enumerate}
If $e_2,e_3$ is a basis of $V(\mathcal A)$ and we write the multiplication table of $\mathcal A$ in the basis $e_1,e_2,e_3$, then the $e_1$-components of the products $e_ie_j$ $(2\le i,j\le 3)$ are the values $\varphi(e_i,e_j)$, and the remaining components form the table of the canonical non-unital kernel.
Note that, for a basis $e_1,e_2,e_3$ of $\mathcal A$ with $e_1$ the unit and $e_ie_j=\sum_k\alpha_{ij}^ke_k$, the vectors $e_2,e_3$ span $V(\mathcal A)$ if and only if
\begin{equation}\label{eq:tracecond}
\tau(e_2)=\alpha_{22}^2+\alpha_{23}^3=0\qquad\text{and}\qquad \tau(e_3)=\alpha_{23}^2+\alpha_{33}^3=0.
\end{equation}

\subsection{Traceless $2$-dimensional commutative algebras}
Consider the following $2$-dimensional commutative algebras with a basis $e_2,e_3$:
\begin{longtable}{lcll}
${\bf N}$ &:& \multicolumn{2}{l}{all products are zero;}\\
${\bf A}$ &:& $e_2e_2=e_3;$ & \\
${\bf D}$ &:& $e_2e_2=e_2,$ & $e_2e_3=-e_3;$ \\
${\bf E}$ &:& $e_2e_2=e_3,$ & $e_3e_3=e_2.$
\end{longtable}
\noindent All of them are traceless.
In the notation of the classification of $2$-dimensional algebras \cite{Volkov2019} (with the basis $e_1,e_2$ renamed to $e_2,e_3$), we have 
${\bf A}={\bf A}_3,$
${\bf D}={\bf D}_2(-1,-1)$ and ${\bf E}\cong{\bf E}_1(-1,-1,-1,-1)$: if $\omega$ is a primitive cube root of unity, then 
\begin{center}
$u=e_2+e_3$ and $v=\omega e_2+\omega^2e_3$ satisfy $uu=u$, $uv=-u-v$, $vv=v$.
\end{center}
\begin{lemma}\label{lem:traceless}
Every traceless $2$-dimensional commutative algebra is isomorphic to exactly one of the algebras ${\bf N}$, ${\bf A}$, ${\bf D}$, or ${\bf E}$.
\end{lemma}

\begin{proof}
Let $(V,\mu)$ be a $2$-dimensional commutative algebra and fix a non-zero alternating bilinear form $\varpi$ on $V$ (for a basis $e_2,e_3$ we take $\varpi(e_2,e_3)=1$).
Put \begin{center}
 $f_\mu(x,y,z):=\varpi(\mu(x,y),z)$.
\end{center}
For every linear operator $M$ on $V$ we have \begin{center}$\varpi(My,z)+\varpi(y,Mz)=\tr(M)\,\varpi(y,z)$,
\end{center} since both sides are alternating in $y,z$ and they coincide for $y=e_2$, $z=e_3$.
Applying this to $M=\mu(x,\cdot)$, we get
\[
f_\mu(x,y,z)-f_\mu(x,z,y)=\tr\bigl(\mu(x,\cdot)\bigr)\,\varpi(y,z).
\]
Since $f_\mu$ is symmetric in the first two arguments, $f_\mu$ is a symmetric trilinear form if and only if $\mu$ is traceless.
As $\varpi$ is non-degenerate, $\mu$ is determined by $f_\mu$, and every symmetric trilinear form is of the form $f_\mu$ for some traceless $\mu$.
Hence $\mu\mapsto F_\mu(x):=f_\mu(x,x,x)$ is a bijection between traceless commutative multiplications on $V$ and binary cubic forms on $V$.

For $g\in {\rm GL}(V)$ put $\mu^g(x,y):=g^{-1}\mu(gx,gy)$; then $g\colon(V,\mu^g)\to(V,\mu)$ is an isomorphism, and every multiplication on $V$ isomorphic to $\mu$ is of this form.
Using 
\begin{center}$\varpi(gu,gv)=\det(g)\varpi(u,v)$
\end{center} we obtain
\[
F_{\mu^g}(x)=\det(g)^{-1}F_\mu(gx).
\]
Replacing $g$ by $\lambda g$ multiplies the right-hand side by $\lambda$.
Hence two traceless multiplications on $V$ are isomorphic if and only if their cubic forms belong to the same orbit of the group $\mathbb C^*\times {\rm GL}(V)$ acting by $(c,g)\cdot F=c\,F\circ g$.
A non-zero binary cubic form over $\mathbb C$ is a product of three linear forms, and since ${\rm PGL}_2(\mathbb C)$ acts $3$-transitively on the projective line, its orbit is determined by the multiplicities of its roots: $(3)$, $(2,1)$ or $(1,1,1)$.
Writing $x=ae_2+be_3$, a direct computation gives
\[
F_{\bf N}=0,\qquad F_{{\bf A}}=-a^3,\qquad F_{\bf D}=3a^2b,\qquad F_{\bf E}=b^3-a^3,
\]
which are representatives of the four orbits. This proves the lemma.
\end{proof}

\begin{lemma}\label{lem:aut}
The automorphism groups of the algebras ${\bf N}$, ${\bf A}$, ${\bf D}$, ${\bf E}$ consist of the following linear maps $g$:
\begin{enumerate}[$(1)$]
\item $\Aut({\bf N})\quad =\quad {\rm GL}_2(\mathbb C)$;
\item $\Aut({\bf A})\quad =\quad \{\,g(e_2)=pe_2+qe_3,\ g(e_3)=p^2e_3\ :\ p\in\mathbb C^*,\ q\in\mathbb C\,\}$;
\item $\Aut({\bf D})\quad =\quad \{\,g(e_2)=e_2,\ g(e_3)=se_3\ :\ s\in\mathbb C^*\,\}$;
\item $\Aut({\bf E})\quad =\quad \{\,g(e_2)=\epsilon e_2,\ g(e_3)=\epsilon^2e_3\,\}\cup\{\,g(e_2)=\epsilon e_3,\ g(e_3)=\epsilon^2e_2\,\}$, where $\epsilon^3=1$. 
\end{enumerate}
\end{lemma}

\begin{proof}
Let $g(e_2)=pe_2+qe_3$ and $g(e_3)=re_2+se_3$ with $ps-qr\neq0$.
\begin{enumerate}
\item[For ${\bf A}:$] $g(e_2)g(e_2)=p^2e_3$ must be equal to $g(e_3)$, so $r=0$ and $s=p^2$; then $g(e_2)g(e_3)=0=g(e_3)g(e_3)$ holds automatically.

\item[For ${\bf D}:$] we have $(ae_2+be_3)^2=a^2e_2-2abe_3$. The condition $g(e_3)g(e_3)=0$ gives $r=0$, hence $p\neq0\neq s$; the condition $g(e_2)g(e_2)=g(e_2)$ gives $p^2=p$ and $-2pq=q$, i.e.\ $p=1$ and $q=0$; then $g(e_2)g(e_3)=-se_3=-g(e_3)$.

\item[For ${\bf E}:$] we have $(ae_2+be_3)(ce_2+de_3)=bd\,e_2+ac\,e_3$.
The condition $g(e_2)g(e_2)=g(e_3)$ gives $r=q^2$, $s=p^2$; the condition $g(e_3)g(e_3)=g(e_2)$ gives $p=s^2=p^4$ and $q=r^2=q^4$; the condition $g(e_2)g(e_3)=0$ gives $pr=qs=0$, i.e.\ $pq=0$.
Since $ps-qr=p^3-q^3\neq0$, we get either $q=0$, $p^3=1$, or $p=0$, $q^3=1$.
\end{enumerate}
\end{proof}

\begin{remark}\label{rem:oldkernels}
By \eqref{eq:shift}, the non-unital kernel with respect to an arbitrary complement $W$ determines the canonical one: it is obtained by subtracting the ``trace part'', \begin{center}
$\mu_{V}(x,y)=\mu_W(x,y)-a(x)y-a(y)x$ with $a(x)=\frac13\tr(\mu_W(x,\cdot))$.
\end{center}
For the commutative $2$-dimensional algebras used in the first version of this work (in the notation of \cite{Volkov2019}, basis renamed to $e_2,e_3$) this gives the following types of the canonical non-unital kernel:
\begin{longtable}{lll}
${\bf A}_1(\frac12)$: & $e_2e_2=e_2+e_3,\ e_2e_3=\frac12e_3$ & $\longmapsto\ {\bf A}$;\\
${\bf A}_3$: & $e_2e_2=e_3$ & $\longmapsto\ {\bf A}$;\\
${\bf B}_2(\frac12)$: & $e_2e_3=\frac12e_3$ \ (basis $e_3,e_2$) & $\longmapsto\ {\bf D}$;\\
${\bf C}(\frac12,0)$: & $e_2e_2=e_3,\ e_2e_3=\frac12e_2,\ e_3e_3=e_3$ & $\longmapsto\ {\bf A}$;\\
${\bf D}_1(\frac12,0)$: & $e_2e_2=e_2,\ e_2e_3=\frac12e_2$ & $\longmapsto\ {\bf E}$;\\
${\bf D}_4(\frac12,\frac12)$: & $e_2e_2=e_2,\ e_2e_3=\frac12e_2+\frac12e_3$ & $\longmapsto\ {\bf D}$;\\
${\bf D}_2(\alpha,\alpha)$: & $e_2e_2=e_2,\ e_2e_3=\alpha e_3$ & $\longmapsto\ {\bf D}$ if $\alpha\neq\frac12$,\\
&& $\phantom{\longmapsto}\ \ {\bf N}$ if $\alpha=\frac12$;\\
${\bf E}_1(\alpha,\beta,\alpha,\beta)$: & $e_2e_2=e_2,\ e_3e_3=e_3,$ & $\longmapsto\ {\bf E}$ if $\alpha\neq\frac12\neq\beta$,\\
& $e_2e_3=\alpha e_2+\beta e_3$ & $\phantom{\longmapsto}\ \ {\bf D}$ if exactly one of $\alpha,\beta$ equals $\frac12$,\\
&& $\phantom{\longmapsto}\ \ {\bf N}$ if $\alpha=\beta=\frac12$.
\end{longtable}
\noindent Indeed, the cubic form of the traceless part of ${\bf E}_1(\alpha,\beta,\alpha,\beta)$ is proportional to $ab\bigl((2\beta-1)a-(2\alpha-1)b\bigr)$, and the traceless parts of ${\bf B}_2(\frac12)$ and ${\bf D}_4(\frac12,\frac12)$ are equal to $-\frac13{\bf D}$ in the bases $e_2,e_3$ and $e_3,e_2$, respectively.
The algebra ${\bf D}_4(\frac12,\frac12)$ is the commutative member of the family ${\bf D}_4(\alpha,\beta)$ ($e_1e_1=e_1$, $e_1e_2=(1-\alpha)e_1+\beta e_2$, $e_2e_1=\alpha e_1+(1-\beta)e_2$), which is missing in the classification of \cite{Volkov2019}; it is not isomorphic to any algebra listed there. Since Lemma~\ref{lem:traceless} is proved directly, this gap does not affect the classification below.
In particular, algebras whose non-unital kernels (with respect to some complement) are of different types may be isomorphic, and algebras with non-unital kernels ${\bf E}_1(\alpha,\beta,\alpha,\beta)$ and ${\bf E}_1(\alpha',\beta',\alpha',\beta')$ with $(\alpha,\beta)\neq(\alpha',\beta')$ may be isomorphic as well.
\end{remark}

\section{Main part: classification}

\subsection{Case ${\bf N}$}

\begin{proposition}\label{prop:N}
Let $\mathcal{A}$ be a complex $3$-dimensional unital commutative algebra with the canonical non-unital kernel of type ${\bf N}$.
Then $\mathcal{A}$ is isomorphic to exactly one of the following algebras:
\begin{longtable}{lcl}
$\mathcal{A}_{01}$&$: $&all products $e_ie_j$ $(2\le i,j\le3)$ are zero;\\
$\mathcal{A}_{02}$&$: $&$ e_3e_3=e_1;$\\
$\mathcal{A}_{03}$&$:$&$ e_2e_3 = e_1, \quad e_3e_3=e_1.$
\end{longtable}
\end{proposition}

\begin{proof}
Let $e_2,e_3$ be a basis of $V(\mathcal A)$. Then the multiplication table of $\mathcal A$ has the form
\[
e_2e_2=\alpha e_1,\qquad e_2e_3=\beta e_1,\qquad e_3e_3=\gamma e_1 .
\]
By Lemmas~\ref{lem:canonical} and~\ref{lem:aut}, two such algebras are isomorphic if and only if their Gram matrices
$\bigl(\begin{smallmatrix}\alpha&\beta\\ \beta&\gamma\end{smallmatrix}\bigr)$ are congruent, i.e.\ if and only if these matrices have the same rank.
The Gram matrices of $\mathcal A_{01}$, $\mathcal A_{02}$, $\mathcal A_{03}$ have ranks $0$, $1$, $2$, respectively.
Explicitly, if $\alpha=\beta=\gamma=0$, then we have $\mathcal A_{01}$.
If $\beta^2=\alpha\gamma$ and $(\alpha,\beta,\gamma)\neq(0,0,0)$, then (interchanging $e_2$ and $e_3$ if $\gamma=0$) we may assume $\gamma\neq0$, and the change of basis
$e_2'=\gamma e_2-\beta e_3$, $e_3'=\gamma^{-1/2}e_3$ gives $\mathcal A_{02}$.
If $\beta^2\neq\alpha\gamma$, the form $\varphi$ is non-degenerate, hence congruent to the non-degenerate form
$\bigl(\begin{smallmatrix}0&1\\ 1&1\end{smallmatrix}\bigr)$, and we obtain $\mathcal A_{03}$.
\end{proof}

\subsection{Case ${\bf A}$}

\begin{proposition}\label{prop:A3}
Let $\mathcal{A}$ be a complex $3$-dimensional unital commutative algebra with the canonical non-unital kernel of type ${\bf A}$.
Then $\mathcal{A}$ is isomorphic to one of the following algebras:
\begin{longtable}{lcl}
$\mathcal{A}_{04}$&$:$&$e_2e_2=e_3;$\\
$\mathcal{A}_{05}$&$:$&$e_2e_2=e_1+e_3;$\\
$\mathcal{A}_{06}$&$:$&$e_2e_2=e_3,\quad e_2e_3=e_1;$\\
$\mathcal{A}_{07}^\alpha$&$:$&$e_2e_2=\alpha e_1+e_3,\quad e_3e_3=e_1.$
\end{longtable}
\noindent All listed algebras are non-isomorphic, except $\mathcal{A}_{07}^\alpha\cong\mathcal{A}_{07}^{-\alpha}$.
\end{proposition}

\begin{proof}
Let $e_2,e_3$ be a basis of $V(\mathcal A)$ in which the canonical non-unital kernel is ${\bf A}$.
Then the multiplication table of $\mathcal A$ has the form
\[
e_2e_2=\alpha e_1+e_3,\qquad e_2e_3=\beta e_1,\qquad e_3e_3=\gamma e_1 .
\]
By Lemmas~\ref{lem:canonical} and~\ref{lem:aut}, the isomorphisms between such algebras are exactly the changes of basis
\[
e_1'=e_1,\qquad e_2'=pe_2+qe_3,\qquad e_3'=p^2e_3\qquad(p\neq0),
\]
and they transform the parameters as follows:
\[
\alpha'=p^2\alpha+2pq\beta+q^2\gamma,\qquad \beta'=p^2(p\beta+q\gamma),\qquad \gamma'=p^4\gamma .
\]
In particular, the conditions $\gamma=0$ and $\gamma=\beta=0$ are invariant, and so is the condition $\gamma=\beta=\alpha=0$.
\begin{enumerate}
\item If $\gamma\neq0$, then the choice $p=\gamma^{-1/4}$, $q=-p\beta\gamma^{-1}$ gives $\gamma'=1$, $\beta'=0$ and $\alpha'=(\alpha\gamma-\beta^2)\gamma^{-3/2}$; we obtain $\mathcal A_{07}^{\alpha'}$.
If $\mathcal A_{07}^{\alpha}\cong\mathcal A_{07}^{\alpha'}$, then $p^4=1$ and $p^2q=0$, hence $q=0$, $p^2=\pm1$ and $\alpha'=\pm\alpha$; conversely, $p=\sqrt{-1}$, $q=0$ gives $\mathcal A_{07}^{\alpha}\cong\mathcal A_{07}^{-\alpha}$.
\item If $\gamma=0$ and $\beta\neq0$, then the choice $p=\beta^{-1/3}$, $q=-\frac{p\alpha}{2\beta}$ gives $\alpha'=0$, $\beta'=1$, and we obtain $\mathcal A_{06}$.
\item If $\gamma=\beta=0$ and $\alpha\neq0$, then the choice $p=\alpha^{-1/2}$, $q=0$ gives $\mathcal A_{05}$.
\item If $\gamma=\beta=\alpha=0$, then we have $\mathcal A_{04}$.
\end{enumerate}
\end{proof}

\subsection{Case ${\bf D}$}

\begin{proposition}\label{prop:D}
Let $\mathcal{A}$ be a complex $3$-dimensional unital commutative algebra with the canonical non-unital kernel of type ${\bf D}$.
Then $\mathcal{A}$ is isomorphic to one of the following algebras:
\begin{longtable}{lcl}
$\mathcal{A}_{08}^\alpha$&$:$&$e_2e_2=\alpha e_1+e_2;$\\
$\mathcal{A}_{09}^\alpha$&$:$&$e_2e_2=\alpha e_1+e_2,\quad e_2e_3=e_1;$\\
$\mathcal{A}_{10}^{\alpha,\beta}$&$:$&$e_2e_2=\alpha e_1+e_2,\quad e_2e_3=\beta e_1,\quad e_3e_3=e_1.$
\end{longtable}
\noindent All listed algebras are non-isomorphic, except $\mathcal{A}_{10}^{\alpha,\beta}\cong\mathcal{A}_{10}^{\alpha,-\beta}$.
\end{proposition}

\begin{proof}
Let $e_2,e_3$ be a basis of $V(\mathcal A)$ in which the canonical non-unital kernel is ${\bf D}$.
Then the multiplication table of $\mathcal A$ has the form
\begin{equation}\label{eq:Dcanon}
e_2e_2=ae_1+e_2,\qquad e_2e_3=be_1-e_3,\qquad e_3e_3=ce_1 .
\end{equation}
By Lemmas~\ref{lem:canonical} and~\ref{lem:aut}, the isomorphisms between such algebras are exactly the changes of basis $e_1'=e_1$, $e_2'=e_2$, $e_3'=se_3$ $(s\neq0)$, and they transform the parameters by
\[
(a,b,c)\longmapsto(a,\ sb,\ s^2c).
\]
In the basis $e_1$, $f_2:=\frac13(e_1+e_2)$, $e_3$ the table \eqref{eq:Dcanon} becomes
\begin{equation}\label{eq:Dnice}
f_2f_2=\tfrac{a-2}{9}\,e_1+f_2,\qquad f_2e_3=\tfrac b3\,e_1,\qquad e_3e_3=ce_1 .
\end{equation}
Conversely, for the table \eqref{eq:Dnice} the vectors $e_2=3f_2-e_1$ and $e_3$ satisfy \eqref{eq:tracecond} and give back \eqref{eq:Dcanon}.
Hence $\alpha=\frac{a-2}9$ is an invariant, and the conditions $c=0$ and $b=c=0$ are invariant.
\begin{enumerate}
\item If $c\neq0$, the choice $s=c^{-1/2}$ gives $c'=1$ and \eqref{eq:Dnice} becomes $\mathcal A_{10}^{\alpha,\beta}$ with $\beta=\frac{b}{3\sqrt c}$; the two choices of $\sqrt c$ give $\beta$ and $-\beta$, and these are the only possibilities because $c=c'=1$ forces $s=\pm1$.
\item If $c=0$ and $b\neq0$, the choice $s=3b^{-1}$ gives $b'=3$, and \eqref{eq:Dnice} becomes $\mathcal A_{09}^\alpha$.
\item If $c=b=0$, then \eqref{eq:Dnice} is $\mathcal A_{08}^\alpha$.
\end{enumerate}
\end{proof}

\begin{remark}
With respect to the complement $\langle f_2,e_3\rangle$ used in the tables of $\mathcal A_{08}^\alpha$, $\mathcal A_{09}^\alpha$, $\mathcal A_{10}^{\alpha,\beta}$, the non-unital kernel is ${\bf D}_2(0,0)$: $f_2f_2=f_2$.
\end{remark}

\subsection{Case ${\bf E}$}

\begin{proposition}\label{prop:E}
Let $\mathcal{A}$ be a complex $3$-dimensional unital commutative algebra with the canonical non-unital kernel of type ${\bf E}$.
Then $\mathcal{A}$ is isomorphic to an algebra
\begin{longtable}{lcl}
$\mathcal{A}_{11}^{\alpha,\beta,\gamma}$&$:$&$e_2e_2=\alpha e_1+e_3,\quad e_2e_3=\beta e_1,\quad e_3e_3=\gamma e_1+e_2.$
\end{longtable}
\noindent Moreover, $\mathcal{A}_{11}^{\alpha,\beta,\gamma}\cong\mathcal{A}_{11}^{\alpha',\beta',\gamma'}$ if and only if
\[
(\alpha',\beta',\gamma')=(\epsilon\alpha,\beta,\epsilon^2\gamma)\quad\text{or}\quad(\alpha',\beta',\gamma')=(\epsilon\gamma,\beta,\epsilon^2\alpha)\qquad\text{for some } \epsilon\in\mathbb C,\ \epsilon^3=1;
\]
equivalently, if and only if $\beta'=\beta$, $\alpha'\gamma'=\alpha\gamma$ and $\alpha'^3+\gamma'^3=\alpha^3+\gamma^3$.
\end{proposition}

\begin{proof}
If $e_2,e_3$ is a basis of $V(\mathcal A)$ in which the canonical non-unital kernel is ${\bf E}$, then the table of $\mathcal A$ is $\mathcal A_{11}^{\alpha,\beta,\gamma}$ with $\alpha=\varphi(e_2,e_2)$, $\beta=\varphi(e_2,e_3)$, $\gamma=\varphi(e_3,e_3)$.
By Lemmas~\ref{lem:canonical} and~\ref{lem:aut}, the isomorphisms between such algebras are given by the six elements of $\Aut({\bf E})$.
The change of basis $e_2'=\epsilon e_2$, $e_3'=\epsilon^2e_3$ gives \begin{center}
$e_2'e_2'=\epsilon^2\alpha e_1+e_3'$, $e_2'e_3'=\beta e_1$, $e_3'e_3'=\epsilon\gamma e_1+e_2'$,
\end{center} and the change of basis $e_2'=e_3$, $e_3'=e_2$ gives $(\alpha',\beta',\gamma')=(\gamma,\beta,\alpha)$. This proves the first description.

The quantities $\beta$, $\alpha\gamma$ and $\alpha^3+\gamma^3$ are obviously invariant. Conversely, suppose they coincide for $(\alpha,\beta,\gamma)$ and $(\alpha',\beta',\gamma')$. Then 
\begin{center}$\{\alpha^3,\gamma^3\}=\{\alpha'^3,\gamma'^3\}$ (both are the roots of $X^2-(\alpha^3+\gamma^3)X+(\alpha\gamma)^3$);
\end{center} applying the interchange of $e_2$ and $e_3$ we may assume $\alpha'^3=\alpha^3$ and $\gamma'^3=\gamma^3$. If $\alpha\neq0$, then $\alpha'=\epsilon\alpha$ with $\epsilon^3=1$, and $\gamma'=\alpha\gamma/\alpha'=\epsilon^2\gamma$. If $\alpha=0$, then $\alpha'=0$ and $\gamma'=\epsilon^2\gamma$ for a suitable cube root of unity $\epsilon$.
\end{proof}

\subsection{The main theorem}

\begin{theorem}\label{thm:main}
Let $\mathcal{A}$ be a complex $3$-dimensional unital commutative algebra.
Then $\mathcal{A}$ is isomorphic to one of the following algebras:

\begin{longtable}{lll}
\hline
$\mathcal{A}$ & multiplication table & canonical kernel \\
\hline
$\mathcal{A}_{01}$ & all products $e_ie_j$ $(2\le i,j\le 3)$ are zero & ${\bf N}$\\[0.3em]
$\mathcal{A}_{02}$ & $e_3e_3=e_1$ & ${\bf N}$\\[0.3em]
$\mathcal{A}_{03}$ & $e_2e_3=e_1\quad e_3e_3=e_1$ & ${\bf N}$\\[0.3em]
\hline
$\mathcal{A}_{04}$ & $e_2e_2=e_3$ & ${\bf A}$\\[0.3em]
$\mathcal{A}_{05}$ & $e_2e_2=e_1+e_3$ & ${\bf A}$\\[0.3em]
$\mathcal{A}_{06}$ & $e_2e_2=e_3\quad e_2e_3=e_1$ & ${\bf A}$\\[0.3em]
$\mathcal{A}_{07}^\alpha$ & $e_2e_2=\alpha e_1+e_3\quad e_3e_3=e_1$ & ${\bf A}$\\[0.3em]
\hline
$\mathcal{A}_{08}^\alpha$ & $e_2e_2=\alpha e_1+e_2$ & ${\bf D}$\\[0.3em]
$\mathcal{A}_{09}^\alpha$ & $e_2e_2=\alpha e_1+e_2\quad e_2e_3=e_1$ & ${\bf D}$\\[0.3em]
$\mathcal{A}_{10}^{\alpha,\beta}$ & $e_2e_2=\alpha e_1+e_2\quad e_2e_3=\beta e_1\quad e_3e_3=e_1$ & ${\bf D}$\\[0.3em]
\hline
$\mathcal{A}_{11}^{\alpha,\beta,\gamma}$ & $e_2e_2=\alpha e_1+e_3\quad e_2e_3=\beta e_1\quad e_3e_3=\gamma e_1+e_2$ & ${\bf E}$\\[0.3em]
\hline
\end{longtable}

\noindent Here $\alpha,\beta,\gamma\in\mathbb C$.
All listed algebras are non-isomorphic, except
\[
\mathcal{A}_{07}^\alpha\cong\mathcal{A}_{07}^{-\alpha},\qquad
\mathcal{A}_{10}^{\alpha,\beta}\cong\mathcal{A}_{10}^{\alpha,-\beta},\qquad
\mathcal{A}_{11}^{\alpha,\beta,\gamma}\cong\mathcal{A}_{11}^{\epsilon\alpha,\beta,\epsilon^2\gamma}\cong\mathcal{A}_{11}^{\gamma,\beta,\alpha}\quad(\epsilon^3=1).
\]
\end{theorem}

\begin{proof}
By Lemmas~\ref{lem:canonical} and~\ref{lem:traceless}, the canonical non-unital kernel of $\mathcal A$ is of exactly one of the types ${\bf N}$, ${\bf A}$, ${\bf D}$, ${\bf E}$, and this type is an isomorphism invariant. For the tables of $\mathcal A_{01}$--$\mathcal A_{07}^\alpha$ and $\mathcal A_{11}^{\alpha,\beta,\gamma}$ the vectors $e_2,e_3$ satisfy \eqref{eq:tracecond}, so the canonical kernel can be read off from the table; for $\mathcal A_{08}^\alpha$--$\mathcal A_{10}^{\alpha,\beta}$ it is ${\bf D}$ by the proof of Proposition~\ref{prop:D}. Now the statement follows from Propositions~\ref{prop:N}, \ref{prop:A3}, \ref{prop:D} and~\ref{prop:E}.
\end{proof}

\begin{corollary}
The algebras $\mathcal A_{01},\ldots,\mathcal A_{06}$ are the only algebras in the list without parameters; the families $\mathcal A_{07}^\alpha$, $\mathcal A_{08}^\alpha$, $\mathcal A_{09}^\alpha$ depend on one parameter, $\mathcal A_{10}^{\alpha,\beta}$ on two parameters, and $\mathcal A_{11}^{\alpha,\beta,\gamma}$ on three parameters.
A complete system of invariants of $\mathcal A_{11}^{\alpha,\beta,\gamma}$ is $\bigl(\beta,\ \alpha\gamma,\ \alpha^3+\gamma^3\bigr)$.
\end{corollary}

\begin{remark}
The associative algebras in the list of Theorem~\ref{thm:main} are exactly
\[
\mathcal A_{01}\cong\mathbb C[x,y]/(x,y)^2,\qquad
\mathcal A_{04}\cong\mathbb C[x]/(x^3),\qquad
\mathcal A_{08}^0\cong\mathbb C\times\mathbb C[x]/(x^2),\qquad
\mathcal A_{11}^{0,1,0}\cong\mathbb C\times\mathbb C\times\mathbb C,
\]
in accordance with the classical classification of $3$-dimensional unital commutative associative algebras.
For instance, if $E_1,E_2,E_3$ are the primitive idempotents of $\mathbb C^3$ and $\omega$ is a primitive cube root of unity, then $e_2=E_1+\omega E_2+\omega^2E_3$ and $e_3=E_1+\omega^2E_2+\omega E_3$ satisfy $e_2e_2=e_3$, $e_3e_3=e_2$, $e_2e_3=e_1$.
\end{remark}

\end{document}